\newfont{\bcb}{msbm10}
\newfont{\matb}{cmbx10}
\newfont{\got}{eufm10}

\documentclass[12pt]{amsart}
\usepackage{amsmath, amsthm, amscd, amsfonts, amssymb, latexsym, graphicx, color}
\usepackage[bookmarksnumbered, colorlinks, plainpages, hypertex]{hyperref}

\usepackage[cp1250]{inputenc}

\usepackage{amsmath,amsthm}
\usepackage{amssymb,latexsym}
\usepackage{enumerate}

\newtheorem{theorem}{Theorem}[section]
\newtheorem{lemma}[theorem]{Lemma}
\newtheorem{proposition}[theorem]{Proposition}
\newtheorem{corollary}[theorem]{Corollary}
\theoremstyle{definition}

\theoremstyle{remark}
\newtheorem{remark}[theorem]{Remark}
\numberwithin{equation}{section}

\begin{document}

\title[Definable retractions over separated power series]
      {Definable retractions over complete \\
      fields with separated power series}

\author[Krzysztof Jan Nowak]{Krzysztof Jan Nowak}


\subjclass[2000]{Primary 32P05, 32B20, 14G22; Secondary 14P15,
32S45, 03C10.}

\keywords{complete fields with separated power series, definable
retractions, desingularization of terms, closedness theorem,
embedded resolution of singularities, simple normal crossing
divisors, quasi-rational subdomains, extending continuous
definable functions}

\date{}

\begin{abstract}
Let $K$ be a complete non-Archimedean field $K$ with separated
power series, treated in the analytic Denef--Pas language. We
prove the existence of definable retractions onto an arbitrary
closed definable subset of $K^{n}$, whereby definable
non-Archimedean versions of the extension theorems by
Tietze--Urysohn and Dugundji follow directly. We reduce the
problem to the case of a simple normal crossing divisor, relying
on our closedness theorem and desingularization of terms. The
latter result is established by means of the following tools:
elimination of valued field quantifiers (due to
Cluckers--Lipshitz--Robinson), embedded resolution of
singularities by blowing up (due to Bierstone--Milman or Temkin),
the technique of quasi-rational subdomains (due to
Lipshitz--Robinson) and our closedness theorem.
\end{abstract}

\maketitle

\section{Main result}

Fix a complete, rank one valued field $K$ of equicharacteristic
zero (not necessarily algebraically closed). Denote by $v$,
$\Gamma = \Gamma_{K}$, $K^{\circ}$, $K^{\circ \circ}$ and
$\widetilde{K}$ the valuation, its value group, the valuation
ring, maximal ideal and residue field, respectively. The
multiplicative norm corresponding to $v$ will be denoted by $|
\cdot |$. The $K$-topology on $K^{n}$ is the one induced by the
valuation $v$. The word ''definable'' usually means ''definable
with parameters''.

\vspace{1ex}

In the paper, we shall deal with a separated Weierstrass system
$$ \mathcal{S} := \{ S_{m,n}^{\circ} \} = \{ S_{m,n}^{\circ}(E,K) \} $$
from the paper~\cite{L-R-0} (see
also~\cite[Ex.~4.4.(3)]{C-Lip-0}). Any algebraic extension $L$ of
a complete field containing $K$ carries separated analytic
$\mathcal{S}$-structure. It is the collection $\{ \sigma_{m,n} \}$
of homomorphisms from $S_{m,n}^{\circ}$ to the ring of
$L^{\circ}$-valued functions on $(L^{\circ})^{m} \times (L^{\circ
\circ})^{n}$, which are canonically induced by the inclusion $K
\subset L$. These fields with analytic structure are treated in an
analytic Denef--Pas language $\mathcal{L}$. It is the two sorted,
semialgebraic language $\mathcal{L}_{Hen}$ (with the main, valued
field sort $K$ and the auxiliary $RV$-sort), augmented on the
valued field sort $K$ by the multiplicative inverse $1/x$ (with
$1/0 :=0$) and the names of all functions of the system
$\mathcal{S}$, together with the induced language on the auxiliary
sort $RV$ (cf.~\cite[Section~6.2]{C-Lip-0} and
also~\cite[Section~2]{Now-3}). Power series $f \in
S_{m,n}^{\circ}$ are construed via the analytic
$\mathcal{A}$-structure on their natural domains and as zero
outside them. Note that in the equicharacteristic case, the
induced language on the sort $RV$ coincides with the semialgebraic
inclusion language.

\vspace{1ex}

Consider a quasi-affinoid algebra $A$, i.e.\ $A = S_{m,n}/I$ for
an ideal $I$ of $S_{m,n}$. The space $X =\mathrm{Max}\, (A)$ of
the maximal ideals of $A$ may be regarded as a scheme with the
Zariski topology and as a quasi-affinoid variety $\mathrm{Sp}\,
(A)$, i.e.\ the locally G-ringed space $(X,\mathcal{O}_{X})$ with
the rigid G-topology; $\mathrm{Max}\, (A)$ carries also the
canonical topology induced by the absolute value. Every algebraic
$K$-variety $V$ can be endowed with an analytic structure of a
rigid analytic variety, called the analytification of $V$. The
analytification of a projective $K$-variety is a quasi-compact
rigid analytic variety.

\vspace{1ex}

In this paper, we are mainly interested in the $K$-rational points
$X(K)$ of rigid analytic varieties $X$ with the canonical
topology. For simplicity of notation, we shall usually write $X$
instead of $X(K)$ when no confusion can arise. The main purpose is
the following theorem on the existence of $\mathcal{L}$-definable
retractions onto an arbitrary closed $\mathcal{L}$-definable
subset, whereby definable non-Archimedean versions of the
extension theorems by Tietze--Urysohn and Dugundji follow
directly.


\begin{theorem}\label{main-2}
Let $Z \subset W$ be closed subvarieties of the unit balls
$(K^{\circ})^{N}$ or $(K^{\circ \circ})^{N}$, of the projective
space $\mathbb{P}^{n}(K)$ or of the products $(K^{\circ})^{N}
\times \mathbb{P}^{n}(K)$ or $(K^{\circ \circ})^{N} \times
\mathbb{P}^{n}(K)$, and let $X:=W \setminus Z$. Then, for each
closed $\mathcal{L}$-definable subset $A$ of $X$, there exists an
$\mathcal{L}$-definable retraction $X \to A$.
\end{theorem}

We immediately obtain

\begin{corollary}\label{main}
For each closed $\mathcal{L}$-definable subset $A$ of $K^{n}$,
there exists an $\mathcal{L}$-definable retraction $K^{n} \to A$.
\end{corollary}

Using embedded resolution of singularities, we reduce the problem
of definable retractions onto a closed $\mathcal{L}$-definable
subset $A$ of $(K^{\circ})^{N}$ to that onto a simple normal
crossing divisor. This is possible via our closedness theorem
(see.~\cite{Now-3} for the analytic and ~\cite{Now-1,Now-2} for
the algebraic versions) and desingularization of
$\mathcal{L}$-terms established in Section~2. The basic tools
applied in our approach are the following:

$\bullet$ elimination of valued field quantifiers in the
      Denef--Pas language $\mathcal{L}$
      (cf.~\cite[Theorem~4.2]{C-Lip-R}
      and~\cite[Theorem~6.3.7]{C-Lip-0});

$\bullet$ the analytic closedness theorem
(cf.~\cite[Theorem~1.1]{Now-3}) which, in particular, enables
      application of resolution of singularities to problems concerning the canonical
      topology;

$\bullet$ embedded resolution of singularities by blowing up for
      quasi-compact rigid analytic spaces (cf.~\cite{BM,Tem-2});

$\bullet$ and the technique of quasi-rational subdomains
      (cf.~\cite{L-R-0}).

\vspace{1ex}

\section{Desingularization of terms}

Consider the ring
$$ S_{m,n} := K \otimes_{K^{\circ}} S_{m,n}^{\circ}, \ \ m,n
   \in \mathbb{N}, $$
of $K$-valued series power series. For $f = \sum a_{\mu,\nu}
\xi^{\mu} \rho^{\nu} \in S_{m,n}$, the Gauss norm
$$ \| f \| = \sup_{\mu,\nu} |a_{\mu,\nu}| = \max_{\mu,\nu}
   |a_{\mu,\nu}| $$
is attained, and one has
$$ S_{m,n}^{\circ} := \{ f \in S_{m,n}: \, \| f \| \leq 1 \} =  $$
$$ \{ f \in S_{m,n}: \, |f^{\sigma}(a,b)| \leq 1 \ \ \text{for all} \
   (a,b) \in (K_{alg}^{\circ})^{m} \times (K_{alg}^{\circ \circ})^{n} \}. $$
Lipshitz--Robinson~\cite{L-R-0} proved that the rings $S_{m,n}$
are regular, excellent, unique factorization domains and satisfy
the Nullstellensatz. Further, they developed a theory of
quasi-affinoid algebras over the rings of separated power series,
including generalized rings of fractions and rational and
quasi-affinoid subdomains. This was done in analogy to the
classical theory of affinoid algebras over the Tate rings of
strictly convergent power series (cf.~\cite{B-G-R}). We still need
the following

\begin{proposition}\label{prop-part}
Let $A$ be a quasi-affinoid algebra, $a,b \in A$ be two elements
such that, at each point $x \in \mathrm{Max}\, A$, one has $aA_{x}
\subset bA_{x}$ or $bA_{x} \subset aA_{x}$. Then there is a
partition of $X = \mathrm{Max}\, A$ into quasi-rational subdomains
$X_{i}$ such that, on each $X_{i}$, $a$ is divisible by $b$ or $b$
is divisible by $a$.
\end{proposition}

\begin{proof}
Put
$$ V_{a}:= V(\mathrm{Supp}\,((a,b)/(a))) = V(\mathrm{Ann}\,(a,b)/(a))) $$
and
$$ V_{b}:= V(\mathrm{Supp}\,((a,b)/(b))) = V(\mathrm{Ann}\,((a,b)/(b)). $$
By the assumption, $V_{1} \cap V_{2} = \emptyset$. Take generators
$f=(f_{1},\ldots,f_{r})$ and $g =(g_{1},\ldots,g_{s})$ of
$\mathrm{Ann}\,((a,b)/(a))$ and $\mathrm{Ann}\,((a,b)/(b))$,
respectively. It follows from the closedness theorem that there is
an $\epsilon \in |K|$, $\epsilon
>0$, such that the two quasi-rational subdomains
$$ X_{a} := \{ x \in \mathrm{Max}\, A: \ |f(x) \leq \epsilon \} \
  \ \text{and} \ \ X_{a} := \{ x \in \mathrm{Max}\, A: \ |f(x) \leq \epsilon \} $$
are disjoint. But the intersection of two quasi-rational
subdomains is a quasi-rational subdomain and the complement of a
quasi-rational subdomain is a finite disjoint union of
quasi-rational subdomains. Hence the assertion follows.
\end{proof}

Let $\lhd$ stands for $<$, $>$ or $=$. We immediately obtain

\begin{corollary}\label{NC-part}
On each subdomain $X_{i}$, the subset
$$ \{ x \in X_{i}: |a(x)| \lhd |b(x)|, \ b(x) \neq 0 \}
$$
is the trace on the subset $\{ x \in X_{i}: b(x) \neq 0 \}$ of the
quasi-rational subdomain of $X_{i}$ (thus being an $R$-subdomain
of $X$)
$$ \left\{ x \in X_{i}: \left| \frac{a(x)}{b(x)} \right| \lhd 1 \right\} \ \ \text{or}
   \ \ \left\{ x \in X_{i}: 1  \lhd \left| \frac{b(x)}{a(x)} \right| \right\}, $$
according as $a$ is divisible by $b$ or $b$ is divisible by $a$ on
$X_{i}$.  \hspace*{\fill} $\Box$
\end{corollary}

\begin{remark}\label{rem-NC-part}
Corollary~\ref{NC-part} can be, of course, generalized to the case
of several pairs $a_{i}(x)$, $b_{i}(x)$ satisfying the above
divisibility condition.
\end{remark}

Similarly, the following strengthening of
Proposition~\ref{prop-part} can be proven.

\begin{proposition}
Let $A$ be a quasi-affinoid algebra, $a_{1},\ldots,a_{k} \in A$ be
$k$ elements such that, at each point $x \in \mathrm{Max}\, A$,
the ideals $a_{i}A_{x}$ are linearly ordered by inclusion. Then
there is a partition of $X = \mathrm{Max}\, A$ into quasi-rational
subdomains $X_{i}$ such that the functions $a_{1}, \ldots,a_{k}$
are linearly ordered by divisibility relation on each $X_{i}$.
\hspace*{\fill} $\Box$
\end{proposition}

\vspace{1ex}

The desingularization of $\mathcal{L}$-terms provided in this
paper will be based on resolution of singularities by blowing up
for quasi-affinoid varieties or, more generally, quasi-compact
rigid analytic varieties. The canonical desingularization by
Bierstone--Milman~\cite{BM} by finite sequences of blow-ups
(multi-blowups for short) along admissible smooth centers applies
directly to quasi-compact rigid analytic spaces $X$ over a
complete field of characteristic zero with non trivial absolute
value; and to restrictions of such spaces to their $K$-rational
points $X(K)$ as well. And so does the functorial
desingularization for quasi-excellent schemes by Temkin
(cf.~\cite[Theorem~5.2.2]{Tem-1} and
\cite[Theorem~~1.1.13]{Tem-2}) via analytification and
functoriality.

\vspace{1ex}

Consider a quasi-affinoid variety $X = \mathrm{Max}\, A$ and an
ideal $I$ of $A$ generated by $(n+1)$ analytic functions on $X$.
Then the blow-up of $X$ along the subvariety $V(I)$ induced by $I$
is a quasi-compact rigid analytic subvariety of the projective
variety
$$ \mathbb{P}^{n}(X) = X \times_{K} \mathbb{P}^{n}(K); $$
the projective rigid analytic space $\mathbb{P}^{n}(K)$ is the
analytification of the projective space over $K$
(cf.~\cite[Example~9.3.4/3]{B-G-R} for a natural construction in
the classical affinoid case). By convention, we regard identity
map as a multi-blowup.

\vspace{1ex}

\begin{proposition}\label{ind}
Let $X$ be a quasi-compact irreducible rigid analytic variety,
$f_{1},\ldots,f_{p},g_{1} \neq 0, \ldots,g_{q} \neq 0$ be analytic
functions on $X$ and $P$ a polynomial in $p+q$ indeterminates with
coefficients from $K$. Then there exist an admissible multi-blowup
$\sigma: \widetilde{X} \to X$ and a finite partition of
$\widetilde{X}$ into R-subdomains $\widetilde{X}_{i}$ of
$\widetilde{X}$ such that, on each subset
$$ \{ x \in \widetilde{X}_{i}:  g_{1}(x) \cdot \ldots \cdot g_{q}(x) \neq 0 \}, $$
the pull-back
$$ P(f_{1},\ldots,f_{p},1/g_{1}, \ldots, 1/g_{q}) \circ \sigma
$$
is the restriction of an analytic function $\omega$ on
$\widetilde{X}_{i}$ or its multiplicative inverse and, moreover,
we have $|\omega| \lhd 1$, where $\lhd$ is either $<$, $>$ or $=$.
\end{proposition}

\begin{proof}
Obviously, we have
\begin{equation}\label{eq1}
   P(f_{1},\ldots,f_{p},1/g_{1}, \ldots, 1/g_{q}) = \frac{Q(f_{1},\ldots,f_{p},g_{1},\ldots,
   g_{q})}{(g_{1} \cdot \ldots \cdot g_{q})^{d}}
\end{equation}
for some polynomial $Q$ and a positive integer $d$. Take an
admissible multi-blowup $\sigma: \widetilde{X} \to X$ such that
the pull-backs under $\sigma$ of the numerator and denominator of
the above fraction are simple normal crossing divisors (unless
$Q(f_{1},\ldots,f_{p},g_{1},\ldots, g_{q})$ vanishes, yet which is
a trivial case) being locally (in the Zariski topology) linearly
ordered with respect to divisibility relation. Then the conclusion
follows directly from Corollary~\ref{NC-part}.
\end{proof}

\begin{remark}
The multi-blowup $\sigma$ is an isomorphism of (non quasi-compact)
rigid analytic varieties over the complement of the zero locus of
the function
$$ g_{1} \cdot \ldots \cdot g_{q} \cdot
   Q(f_{1},\ldots,f_{p},g_{1},\ldots, g_{q}). $$
\end{remark}

It is not difficult to strengthen the above proposition as
follows.

\begin{corollary}\label{ind2}
The conclusion of Proposition~\ref{ind} holds in the case of
several polynomials $P$ and several tuples of analytic functions.
\hspace*{\fill} $\Box$
\end{corollary}

\begin{remark}\label{composite}
Consider a quasi-affinoid algebra $A$ and elements
$$ f_{1},\ldots,f_{m},g_{1},\ldots,g_{n} \in A. $$
Then one has a unique homomorphism
$$ S_{m,n} \ni \phi \longrightarrow \phi
   (f_{1},\ldots,f_{m},g_{1},\ldots,g_{n}) \in A\langle f \rangle
   [[g]] $$
that sends $\xi_{i} \mapsto f_{i}$ and $\rho_{j} \mapsto g_{j}$
for $i=1,\ldots,m$, $j=1,\ldots,n$ (cf.~\cite[Corollary~5.1.8 and
Lemma~5.2.2]{L-R-0}). Similarly, given a quasi-compact rigid
analytic variety $X$, elements
$$ f_{1},\ldots,f_{m},g_{1},\ldots,g_{n} \in \mathcal{O}_{X}(X) $$
determine a unique homomorphism
$$ S_{m,n} \ni \phi \longrightarrow \phi
   (f_{1},\ldots,f_{m},g_{1},\ldots,g_{n}) \in \mathcal{O}_{X}(U), $$
where $U := \{ x \in X: |f(x)| \leq 1, \ |g(x)| < 1 \}$.
\end{remark}

\vspace{1ex}

Repeated application of Corollary~\ref{ind2} enables a
desingulariztion for $\mathcal{L}$-terms restricted to the closed
unit balls $\mathbb{B}_{N}$ via an inductive process described as
follows. We shall proceed with induction on the degree $\deg t$ of
$\mathcal{L}$-terms $t$, i.e.\ the maximum number of nested
superpositions that occur in $t$. Any $\mathcal{L}$-term $t$ of
degree $0$ in $N$ variables is on the unit ball $\mathbb{B}_{N}
\subset K^{N}$ of the form $P(f_{1},\ldots,f_{p},1/g_{1}, \ldots,
1/g_{q})$ from Corollary~\ref{ind2} with $X=\mathbb{B}_{N}$.


$\mathcal{L}$-terms $t$ of the form $h(t_{1},\ldots,t_{N})$, for a
function $h \in S_{m,n}$ with $m+n=N$ and some $\mathcal{L}$-terms
$t_{1},\ldots,t_{N}$ of degree $0$, are of degree $1$. Application
of Corollary~\ref{ind2} to the terms $t_{1},\ldots,t_{N}$ gives
rise to a stratification of $X_{0} := \mathbb{B}_{N}$ into Zariski
locally closed subsets $X_{0,i}$ resulting from
equations~\ref{eq1} which occur in given terms; namely one should
consider all equalities or inequalities of the form:
\begin{equation}\label{eq2}
   g_{1} = \ldots g_{k}=0, \ \ g_{k+1} \cdot \ldots \cdot g_{n} \neq 0,
\end{equation}
and
$$ Q(f_{1},\ldots,f_{p},g_{k+1},\ldots, g_{q}) = 0 \ \ \text{or} \ \
   Q(f_{1},\ldots,f_{p},g_{k+1},\ldots, g_{q}) \neq 0, $$
where $Q$ is a polynomial and $d$ a positive integer such that
$$
   P(f_{1},\ldots,f_{p},0, \ldots,0, 1/g_{k+1},\ldots 1/g_{q}) =
   \frac{Q(f_{1},\ldots,f_{p},g_{k+1},\ldots,
   g_{q})}{(g_{k+1} \cdot \ldots \cdot g_{q})^{d}}.
$$
The further data of the process on this first stage are the
restrictions $\sigma_{1i}$ to $\widetilde{X_{0i}} =
\sigma^{-1}(X_{0i})$ of the admissible multi-blowup $\sigma$ of
$X_{0}$ from Proposition~\ref{ind} if the inequality
$$ Q(f_{1},\ldots,f_{p},g_{k+1},\ldots, g_{q}) \neq 0 $$
holds on $X_{0,i}$; otherwise let $\sigma_{1i}$ be the identity on
$\widetilde{X_{0,i}} = X_{0,i}$. Note that every stratum
$\widetilde{X_{0,i}}$ is disjoint with the exceptional divisor of
the multi-blowup $\sigma$. Finally, partition each
$\widetilde{X_{0,i}}$ by means of $R$-subdomains from the
conclusion of Proposition~\ref{ind}. Let $X_{1} := \coprod
X_{1,i_{1}}$ be the disjoint union of all subsets $X_{1,i_{1}}$ of
those partitions (all $\widetilde{X_{0,i}}$ are taken into
account), being (non quasi-compact) rigid analytic varieties, and
$\sigma_{1}: X_{1} \to X_{0}$ be the induced map.

Then, on each $X_{1,i_{1}}$, every pull-back $t_{i} \circ
\sigma_{1}$ is an analytic function or its multiplicative inverse.
In view of Remark~\ref{composite}, the remaining inequalities from
Proposition~\ref{ind} ensure that the pull-back
$h(t_{1},\ldots,t_{N}) \circ \sigma_{1}$ is an analytic function
on each $X_{1,i_{1}}$.

Any $\mathcal{L}$-term $t$ of degree $1$ can be expressed
polynomially by terms considered above or their multiplicative
inverses. As before, application of Corollary~\ref{ind2} yields
the second stage of the process with the data: $\sigma_{2}: X_{2}
\to X_{1}$ and $X_{2} := \coprod X_{2,i_{1},i_{2}}$ the disjoint
union of (non quasi-compact) rigid analytic varieties
$X_{1,i_{1},i_{2}}$ constructed as at the first stage. Then the
pull-back $t \circ \sigma_{1} \circ \sigma_{2}$ is an analytic
function on $X_{2}$.

Given an $\mathcal{L}$-term $t$ of degree $k$, the process
described above consists of $(k+1)$ stages which yield $(k+1)$
maps $\sigma_{j}: X_{j} \to X_{j-1}$, $j=1,\ldots,k+1$,  induced
by admissible multi-blowups, such that the pull-back
$$ t \circ \sigma_{1} \circ \sigma_{2} \circ \ldots \circ \sigma_{k+1} $$
of the term $t$ ia an analytic function on $X_{k}$.

\vspace{1ex}

Note that the maps $\sigma_{j}$ are not definably closed because
stratifications are involved. The restrictions of multi-blowups,
of which the maps $\sigma_{j}$ are built, are in a sense not
linked together. Therefore we need another variant which, however,
does not cover all points of $X_{0} =\mathbb{B}_{N}$.
Nevertheless, it enables induction with respect to the dimension
$N$ of the ambient space. Namely we shall not stratify the space
with respect to equalities and inequalities \ref{eq2}, but first
consider a global multi-blowup $\sigma_{1}: \widetilde{X_{0}} \to
X_{0}$ and next partitioning from Proposition~\ref{ind} and
Corollary~\ref{ind2}, applied to the polynomials
$P(f_{1},\ldots,f_{p},1/g_{1}, \ldots, 1/g_{q})$. This allows us
to control the term on the complement of the zero locus of the
functions
\begin{equation}\label{eq-fun}
g_{1} \cdot \ldots \cdot g_{q} \cdot
Q(f_{1},\ldots,f_{p},g_{1},\ldots,
   g_{q}),
\end{equation}
over which the multi-blowup is an isomorphism of (non
quasi-compact) rigid analytic varieties. Therefore repeated
application of Corollary~\ref{ind2} enables a desingulariztion for
$\mathcal{L}$-terms restricted to the closed unit balls
$\mathbb{B}_{N}$ via the following inductive process of alternate
admissible multi-blowups and partitions into R-subdomains.

\begin{proposition}\label{gen}
Let $t$ be an $\mathcal{L}$-term of degree $k$. Then there exists
a desingularization process for the term $t$ restricted to the
unit ball $\mathbb{B}_{N}$ which consists of the following data:

\vspace{1ex}

$\bullet$ $X_{0} := \mathbb{B}_{N}$, $\sigma_{1}:
\widetilde{X_{0}} \to X_{0}$ an admissible multi-blowup, $\{
X_{1,i_{1}} \}_{i_{1}}$ a finite partition of $\widetilde{X_{0}}$
into $R$-subdomains;

$\bullet$ $X_{1} := \coprod X_{1,i_{1}}$ the disjoint union,
$\sigma_{2}: \widetilde{X_{1}} \to X_{1}$ an admissible
multi-blowup, $\{ X_{2,i_{1},i_{2}} \}_{i_{1},i_{2}}$ a finite
partition of $\widetilde{X_{1}}$ into $R$-subdomains compatible
with each $\sigma_{2}^{-1}(X_{1,i_{1}})$;

$$ \ldots\ldots\ldots $$

$\bullet$ $X_{k} := \coprod X_{k,i_{1}},\ldots,i_{k}$ the disjoint
union, $\sigma_{k+1}: \widetilde{X_{k}} \to X_{k}$ an admissible
multi-blowup, $\{ X_{k+1,i_{1},\ldots,i_{k+1}}
\}_{i_{1},\ldots,i_{k+1}}$ a finite partition of
$\widetilde{X_{k}}$ into $R$-subdomains compatible with each
$\sigma_{k+1}^{-1}(X_{k,i_{1},\ldots,i_{k}})$;

$\bullet$ $X_{k+1} := \coprod X_{k+1,i_{1},\ldots,i_{k+1}}$ the disjoint union;

$\bullet$ analytic functions $\chi_{j}$, $\psi_{j}$ on $X_{j}$ and
$\phi_{j}$ on $\widetilde{X_{j}}$, $\phi_{j}$ being simple normal
crossing divisors (unless non-vanishing), $j=0,\ldots,k$, defined
recursively as follows: the functions $\chi_{0} = \psi_{0},
\chi_{1},\ldots,\chi_{k}$ are provided successively (they are
products of factors of the form~\ref{eq-fun}), next
$$ \phi_{j} := \psi_{j} \circ \sigma_{j+1}, \ \ \psi_{j+1} :=
   \phi_{j} \cdot \chi_{j+1}, \ \ j=0,1,\ldots,k, $$
and eventually $\phi_{k} := \psi_{k} \circ \sigma_{k+1}$.

$\bullet$ an analytic function $\omega_{k}$ on the complement of
the zero locus $V(\phi_{k})$ of $\phi_{k}$ such that, on each
$R$-subdomain $X_{k+1,i_{1},\ldots,i_{k+1}}$, $\phi_{k}$ and
$\omega_{k}$ or $\phi_{k}$ and $1/\omega_{k}$ are simultaneous
simple normal crossing divisors, unless $\omega_{k}$ vanishes.

\vspace{1ex}

The spaces $X_{j}$ and $\widetilde{X_{j-1}}$, $j=1,\ldots,k$, are
homeomorphic in the canonical topology. The multi-blowups $\sigma_{j}$ transform $\psi_{j-1}$ to normal crossing divisors and are homeomorphisms
over the complement of $V(\psi_{j-1})$ for $j=1,\ldots,k+1$.  By abuse of notation, we
shall regard $\sigma_{j}$ also as a map from $X_{j}$ onto $X_{j-1}$.

\vspace{1ex}

The process results in that the pull-back
$$ t \circ \sigma_{1} \circ \sigma_{2} \circ \ldots \circ \sigma_{k+1} $$
is the restriction of $\omega_{k}$ away from the zero locus
$V(\phi_{k})$ of $\phi_{k}$.  \hspace*{\fill} $\Box$
\end{proposition}

\vspace{1ex}

It is not difficult to obtain the following generalization.

\begin{corollary}\label{gen2}
The conclusion of Proposition~\ref{gen} holds in the case of
several $\mathcal{L}$-term restricted to the unit ball
$\mathbb{B}_{N}$.  \hspace*{\fill} $\Box$
\end{corollary}

\begin{remark}\label{rem-rat}
Resolution of singularities including, in particular, the
technique of admissible blow-ups, applies to the $K$-rational
points $X(K)$ of rigid analytic varieties $X$. Therefore also
valid are the counterparts of the above two results for
$K$-rational points, whose formulations are straightforward. For
instance, the closed unit ball $\mathbb{B}_{N}$ must be replaced
with $\left( K^{\circ} \right)^{N}$.
\end{remark}


By the closedness theorem (cf.~\cite[Theorem~1.1]{Now-3}) and the
descent property below, the technique of blowing up can be applied
to constructions of definable retractions.

\begin{lemma}\label{descent}
Let $\sigma: Y \to X$ be a continuous $\mathcal{L}$-definable
surjective map which is $\mathcal{L}$-definably closed, and $A$ be
a closed $\mathcal{L}$-definable subset of $X$ such that $\sigma$
is bijective over $X \setminus A$. Then every
$\mathcal{L}$-definable retraction $\widetilde{r}: Y \to
\sigma^{-1}(A)$ descends uniquely to a unique
$\mathcal{L}$-definable retraction $r: X \to A)$.
\end{lemma}

\begin{proof}
The retraction $r$ is given by the formula
(cf.~\cite[Lemma~3.1]{Now-4}):
$$ r(x) = \left\{ \begin{array}{cl}
                        \sigma \left( \tilde{r} \left( \sigma^{-1}(x) \right) \right) & \mbox{ if } \ x \in X \setminus A, \\
                        x & \mbox{ if } \ x \in A.
                        \end{array}
               \right.
$$
\end{proof}

We immediately obtain the two corollaries.

\begin{corollary}\label{descent-2}
Let $\sigma: Y \to X$ be a continuous $\mathcal{L}$-definable
surjective map which is $\mathcal{L}$-definably closed, $A$ and
$S$ be closed $\mathcal{L}$-definable subsets of $X$ and $W:=
\sigma^{-1}(S)$. Suppose that the restriction of $\sigma$ to $Y
\setminus W$ is bijective. Then every $\mathcal{L}$-definable
retraction $\widetilde{r}: Y \to \sigma^{-1}(A) \cup W$ descends
uniquely to an $\mathcal{L}$-definable retraction $r: X \to A \cup
S$. \hspace*{\fill} $\Box$
\end{corollary}

\begin{corollary}\label{descent-3}
Under the above assumptions, suppose that there is an
$\mathcal{L}$-definable retraction $\theta: S \to A \cap S$. Then
there is an $\mathcal{L}$-definable retraction $\rho: X \to A$.
\end{corollary}

\begin{proof}
Define the map $\eta: A \cup S \to A$ by putting $\eta (x) =
\theta(x)$ if $x \in S$ and $\eta (x) = x$ if $x \in A$.
Obviously, $\eta$ is an $\mathcal{L}$-definable retraction. Next
set $\rho := \eta \circ r$.
\end{proof}

\section{Proof of the main theorems on definable retractions}

In this section, we shall deal with the rational points $X(K)$ of
rigid analytic varieties $X$. For simplicity of notation, from now
on we shall write it simply $X$ instead of $X(K)$.

\vspace{1ex}

We begin by considering definable retractions of closed definable
subsets contained in the closed unit ball $X_{0} =
(K^{\circ})^{N}$. However all the arguments carry over, mutatis
mutandi, to subsets in the open unit ball $(K^{\circ \circ})^{N}$,
in the projective space $\mathbb{P}^{n}(K)$ or in the products
$(K^{\circ})^{N} \times \mathbb{P}^{n}(K)$ or $(K^{\circ
\circ})^{N} \times \mathbb{P}^{n}(K)$. We shall proceed with
induction on the dimension $N$ of the ambient space $X_{0}$. To
this end, we need the following

\begin{lemma}\label{ind-var}
Let $Z \varsubsetneq X$ be two closed subvarieties of the product
$(K^{\circ})^{N} \times \mathbb{P}^{n}(K)$ and $A$ a closed
$\mathcal{L}$-definable subset of $Z$. Suppose that $X$ is
non-singular of dimension $N$ and Theorem~\ref{main-2} holds for
closed $\mathcal{L}$-definable subsets of every non-singular
variety of this kind of dimension $< N$. Then there exists an
$\mathcal{L}$-definable retraction $r:Z \to A$.
\end{lemma}

\begin{proof}
We shall proceed with induction on the dimension of $Z$. First
apply embedded resolution of singularities and take a multi-blowup
$\tau: V \to X$ which is an isomorphism over the complement of the
singular locus $S$ of $Z$ and such that the pre-image
$\tau^{-1}(Z)$ is a simple normal crossing subvariety of $V$,
which is the union of the transform $\widetilde{Z}$ of $Z$ and the
exceptional divisor $E$.

Taking the disjoint union $Y$ of $\widetilde{Z}$ and the
components of $E$, we get a map $\sigma: Y \to Z$ which satisfies
the assumptions from Corollary~\ref{descent-2} with $W :=
\sigma^{-1}(S)$. Of course, $Y$ is non-singular of dimension $<N$.
By the assumption, there is an $\mathcal{L}$-definable retraction
$\tilde{\rho}: Y \to \sigma^{-1}(A) \cup W$. Hence and by
Corollary~\ref{descent-2}, there is an $\mathcal{L}$-definable
retraction
$$ \rho: Z \to A \cup S. $$

But, by the induction hypothesis, there is an
$\mathcal{L}$-definable retraction
$$ r_{1}: S \to A \cap S $$
which, similarly as it was in the proof of Lemma~\ref{descent-3},
gives rise to an $\mathcal{L}$-definable retraction
$$ r_{2}: A \cup S \to A. $$
Then the map $r := r_{2} \circ \rho$ is the retraction we are
looking for.
\end{proof}

\vspace{1ex}

By elimination of valued field quantifiers, every
$\mathcal{L}$-definable subset of $(K^{\circ})^{N}$ is a finite
union of sets of the form
$$ A = \left\{ x \in (K^{\circ})^{N}: (\mathrm{rv}\, t_{1}(x), \ldots, \mathrm{rv}\, t_{s}(x))
   \in B \right\}, $$
where $B$ is an $\mathcal{L}$-definable subset of $RV(K)^{s}$. We
are going to combine this description with the desingularization
process from Proposition~\ref{gen} and Corollary~\ref{gen2}
applied to the sets of $K$-rational points, as indicated in
Remark~\ref{rem-rat}. Putting
$$ A^{\sigma} := \sigma^{-1}(A), \ \ t^{\sigma} := t \circ \sigma \ \ \text{and} \ \ \tau_{j} := \sigma_{1} \circ \ldots \circ \sigma_{j}, $$
we get
$$ A^{\tau_{k+1}} \cap X_{k+1,i_{1},\ldots,i_{k+1}} = $$
$$ \left\{ x \in X_{k+1,i_{1},\ldots,i_{k+1}}:
   (\mathrm{rv}\, t_{1}^{\tau_{k+1}}(x), \ldots, \mathrm{rv}\, t_{s}^{\tau_{k+1}}(x)) \in B \right\} $$
and
\begin{equation}\label{set1}
 (A^{\tau_{k+1}} \cap X_{k+1,i_{1},\ldots,i_{k+1}}) \setminus
   V(\phi_{k}) =
\end{equation}
$$ \left\{ x \in X_{k+1,i_{1},\ldots,i_{k+1}} \setminus V(\phi_{k}):
   (\mathrm{rv}\, \omega_{1}(x), \ldots, \mathrm{rv}\, \omega_{s}(x)) \in B \right\}. $$
Clearly, modifying the set $B$, we can assume that the functions
$\phi_{k}$ and $\omega_{1},\ldots,\omega_{s}$ are simultaneous
normal crossing divisors on $X_{k+1,i_{1},\ldots,i_{k+1}}$. Then
the set~\ref{set1} is a finite union of the subsets
$$ (X_{k+1,i_{1},\ldots,i_{k+1}} \setminus V(\phi_{k})) \cap V_{j} \cap G_{j}, $$
where
$$ V_{j} := \{ x \in X_{k+1,i_{1},\ldots,i_{k+1}}: \omega_{i}(x)
   =0 \ \ \text{for} \ \ i \in I \} $$
and
$$ G_{j} := \left\{ x \in X_{k+1,i_{1},\ldots,i_{k+1}} \setminus V(\phi_{k}) :
   (\mathrm{rv}\, \omega_{1}(x), \ldots, \mathrm{rv}\, \omega_{s}(x)) \in B, \right. $$
$$ \left. \omega_{q}(x) \neq 0 \ \ \text{for} \ \ q \in \{ 1,\ldots,s \} \setminus I \right\} $$
with all subsets $I \subset \{1,\ldots,s \}$. It is easy to check
that $G_{j}$ are clopen subsets of $X_{k+1,i_{1},\ldots,i_{k+1}}
\setminus V(\phi_{k})$.

\vspace{1ex}

Hence the set
$$ (A^{\tau_{k+1}} \cup V(\phi_{k})) \cap X_{k+1,i_{1},\ldots,i_{k+1}} $$
falls under the description from Proposition~3.3 of our
paper~\cite{Now-4}, which treated the algebraic case of the
problems under study. The proof of Theorem~3.1 (op.cit.) on the
existence of definable retractions onto closed definable subsets
uses Proposition~3.2 (op.cit.) on the existence of definable
retractions onto Zariski closed subsets and Proposition~3.3
(op.cit.). Note that the proofs of those Proposition~3.2 and
Theorem~3.1 modulo the description from Proposition~3.3 carry over
to the affinoid case treated here.

\vspace{1ex}

Under the circumstances, almost the same proof as in the algebraic
case provides an $\mathcal{L}$-definable retraction
$$ r_{k+1,i_{1},\ldots,i_{k+1}}: X_{k+1,i_{1},\ldots,i_{k+1}} \to
   (A^{\tau_{k+1}} \cup V(\phi_{k})) \cap X_{k+1,i_{1},\ldots,i_{k+1}}. $$
Since the subsets $X_{k+1,i_{1},\ldots,i_{k+1}}$ indexed by
$i_{k+1}$ are a clopen covering of
$\sigma_{k+1}^{-1}(X_{k,i_{1},\ldots,i_{k}})$, we get by gluing
also a retraction
$$ \tilde{r}_{k,i_{1},\ldots,i_{k}}:
   \sigma_{k+1}^{-1}(X_{k,i_{1},\ldots,i_{k}}) \to (A^{\tau_{k+1}} \cup V(\phi_{k})) \cap
   \sigma_{k+1}^{-1}(X_{k,i_{1},\ldots,i_{k}}). $$
Hence and by Lemma~\ref{descent-2}, we get a retraction
$$ r_{k,i_{1},\ldots,i_{k}}: X_{k,i_{1},\ldots,i_{k}} \to
   (A^{\tau_{k}} \cup V(\psi_{k})) \cap X_{k,i_{1},\ldots,i_{k}}.  $$
As before, we get a retraction
$$ \tilde{r}_{k-1,i_{1},\ldots,i_{k-1}}:
   \sigma_{k}^{-1}(X_{k-1,i_{1},\ldots,i_{k-1}}) \to (A^{\tau_{k}} \cup V(\psi_{k})) \cap
   \sigma_{k}^{-1}(X_{k-1,i_{1},\ldots,i_{k-1}}); $$
we have $V(\psi_{k})) = \sigma_{k}^{-1}(V(\psi_{k-1})) \cup
V(\chi_{k})$.

\vspace{1ex}

We shall now proceed with induction on the dimension $N$ of the
ambient space. By Lemma~\ref{ind-var}, we can find a retraction
$$ (\sigma_{k}^{-1}(V(\psi_{k-1})) \cup V(\chi_{k})) \cap \sigma_{k}^{-1}(X_{k-1,i_{1},\ldots,i_{k-1}})
   \to $$
$$ \sigma_{k}^{-1}(V(\psi_{k-1})) \cup (A^{\tau_{k}} \cap V(\chi_{k}))
   \cap \sigma_{k}^{-1}(X_{k-1,i_{1},\ldots,i_{k-1}}).
$$
Since $\sigma_{k}$ is a homeomorphism over the complement of
$V(\psi_{k-1})$, the above two retractions along with
Lemma~\ref{descent-3} yield a retraction
$$ r_{k-1,i_{1},\ldots,i_{k-1}}: X_{k-1,i_{1},\ldots,i_{k-1}} \to (A^{\tau_{k-1}} \cup V(\psi_{k-1})) \cap
   X_{k-1,i_{1},\ldots,i_{k-1}}. $$

We can continue the reasoning along this pattern to eventually
achieve an $\mathcal{L}$-definable retraction $r_{0}: X_{0} \to A$
we are looking for.

\vspace{2ex}

We have thus proven Theorem~\ref{main} on the existence of
$\mathcal{L}$-definable retractions onto any closed
$\mathcal{L}$-definable subset $A$ of the closed unit ball $X_{0}
:= (K^{\circ})^{N}$.

\begin{remark}\label{rem-1}
The above results, both desingularization of terms and retractions
onto definable subsets of $X_{0} := (K^{\circ})^{N}$, will run
almost in the same way when the space $X_{0}$ is the open unit
ball $(K^{\circ \circ})^{N}$, the projective space
$\mathbb{P}^{n}(K)$ or the products $(K^{\circ})^{N} \times
\mathbb{P}^{n}(K)$ or $(K^{\circ \circ})^{N} \times
\mathbb{P}^{n}(K)$. Moreover, definable retractions will exist
when $X_{0}$ is the complement of a closed subvariety $Z$ in these
spaces. For the last assertion, one must require at the final
stage of the desingularization process that $\phi_{k}$,
$\omega_{k}$ and also the pre-image $\tau_{k}^{-1}(Z)$ be
simultaneous simple normal crossing divisors; this corresponds to
the passage from the projective space $\mathbb{P}^{N}(K)$ to
$K^{N}$ indicated in~\cite[Remark~2.10]{Now-4}. Next one must
apply the corresponding versions of Lemma~\ref{descent} along with
its two corollaries and of Lemma~\ref{ind-var}, which take into
account the subvariety $Z$; modifications and proofs of those
versions being straightforward.
\end{remark}

\vspace{1ex}

The foregoing proof of a special case of Theorem~\ref{main-2}
along with Remark~\ref{rem-1} establish its full version, and thus
the proof is complete.   \hspace*{\fill} $\Box$

\vspace{2ex}

We conclude with the following comments.

\begin{remark}
The existence of definable retractions onto closed definable subsets yields a non-Archimedean definable analogue of the Dugundji theorem on the existence of a linear (and continuous) extender and, a fortiori, a non-Archimedean definable analogue of the Tietze--Urysohn extension theorem. These issues in the
algebraic case were discussed in our previous paper~\cite{Now-4}.
\end{remark}

\begin{remark}
The non-Archimedean extension problems for definable functions require a different approach and other techniques in comparison with the classical non-Archi\-medean purely topological ones because, among others, geometry over non-locally compact Henselian valued fields suffers from lack of definable Skolem functions. On the other hand, definability in a suitable language often makes the subject under study tamer and enables application of new tools and techniques. Let us finally emphasize that our interest to problems of non-Archimedean geometry was inspired by our joint paper~\cite{K-N}.
\end{remark}

\begin{remark}
In our subsequent paper~\cite{Now-5}, we establish Theorem~\ref{main-2} in the general settings of Henselian valued fields with analytic structure. Its proof relies on the definable version of canonical desingularization developed in that paper (and carried out within a category of definable, {\em strong} analytic manifolds and maps). Note that the theory of Henselian fields with analytic structure, unifying many earlier approaches within non-Archimedean analytic geometry, was developed in the papers~\cite{L-R,C-Lip-R,C-Lip-0,C-Lip}.
\end{remark}

\vspace{3ex}

\begin{small}
Institute of Mathematics

Faculty of Mathematics and Computer Science

Jagiellonian University


ul.~Profesora S.\ \L{}ojasiewicza 6

30-348 Krak\'{o}w, Poland

{\em E-mail address: nowak@im.uj.edu.pl}
\end{small}

\end{document}